\documentclass[10pt]{article} % For LaTeX2e
\usepackage[preprint]{tmlr}

\usepackage{amsmath,amsfonts,bm}

\def\eqref#1{equation~\ref{#1}}
\def\1{\bm{1}}

\DeclareMathAlphabet{\mathsfit}{\encodingdefault}{\sfdefault}{m}{sl}
\SetMathAlphabet{\mathsfit}{bold}{\encodingdefault}{\sfdefault}{bx}{n}

\usepackage{hyperref}
\usepackage{url}
\usepackage{amsmath, amssymb, amsfonts, latexsym, mathtools}
\usepackage{amsthm}
\usepackage{algorithm}
\usepackage{bm}
\usepackage{booktabs}
\usepackage{wrapfig}

\newtheorem{thm}{Theorem}[section]

\newtheorem{lemma}{Lemma}[section]
\newtheorem{cor}{Corollary}[section]

\newtheorem{assum}{Assumption}[section]

\let\classAND\AND
\let\AND\relax
\usepackage{algorithmic}
\let\AND\classAND
\AtBeginEnvironment{algorithmic}{\let\AND\algoAND}

\allowdisplaybreaks

\title{Improved Convergence Rates of Muon Optimizer \\for Nonconvex Optimization}
\usepackage{times}
\author{\name Shuntaro Nagashima \email opti.shun@gmail.com \\
      \addr Meiji University\\
      \AND
      \name Hideaki Iiduka \email iiduka@cs.meiji.ac.jp \\
      \addr Meiji University\\
      }

\def\month{MM}  % Insert correct month for camera-ready version
\def\year{YYYY} % Insert correct year for camera-ready version
\def\openreview{\url{https://openreview.net/forum?id=XXXX}} % Insert correct link to OpenReview for camera-ready version

\begin{document}

\maketitle

\begin{abstract}%
The Muon optimizer has recently attracted attention due to its orthogonalized first-order updates, and a deeper theoretical understanding of its convergence behavior is essential for guiding practical applications; however, existing convergence guarantees are either coarse or obtained under restrictive analytical settings. In this work, we establish sharper convergence guarantees for the Muon optimizer through a direct and simplified analysis that does not rely on restrictive assumptions on the update rule. Our results improve upon existing bounds by achieving faster convergence rates while covering a broader class of problem settings. These findings provide a more accurate theoretical characterization of Muon and offer insights applicable to a broader class of orthogonalized first-order methods.
\end{abstract}

\section{Introduction}
\subsection{Background}
In recent years, optimization algorithms have become a crucial factor influencing the stability of training and the final performance of large-scale deep neural networks (DNNs). In particular, when dealing with deep models and high-dimensional parameter spaces, issues such as sensitivity to learning rate selection and numerical instability tend to become pronounced, motivating the development of practical optimization methods that can effectively mitigate these challenges. Against this backdrop, the Muon (Momentum orthogonalized by Newton-Schulz) optimizer has emerged as a new optimization method that is increasingly being used for large-scale model training.

Muon \citep{jordan2024muon} performs updates based on first-order information, similarly to stochastic gradient descent (SGD) \citep{RobbinsMonro1951} and its momentum variants \citep{POLYAK19641,1570572699326076416,pmlr-v28-sutskever13}, while distinguishing itself by incorporating an orthogonalization operation into the update direction (see \citep{si2025adamuonadaptivemuonoptimizer, zhang2025adagradmeetsmuonadaptive, huang2025limuonlightfastmuon, liu2025fedmuonacceleratingfederatedlearning, gruntkowska2025dropmuonupdatelessconverge, ahn2025dion2simplemethodshrink} for the variants of Muon). Whereas Adam \citep{kingma2015adam} stabilizes training through coordinate-wise adaptive scaling, Muon projects the update matrix onto its orthogonal factor, thereby explicitly controlling the correlation structure of the update direction. This property enables Muon to preserve directional information even in high-dimensional spaces, achieving both numerical stability and efficient optimization. Empirical studies have further reported that Muon has advantages over Adam-type optimizers \citep{kingma2015adam,loshchilov2018decoupled} in terms of both computational efficiency and scalability \citep{liu2025muonscalablellmtraining, wang2025muonoutperformsadamtailend, ai2025practicalefficiencymuonpretraining, mehta2025muontrainingtradeoffslatent}, highlighting its potential as a new optimizer capable of replacing conventional methods.

\subsection{Motivation}
Several attempts \citep{tang2025convergenceanalysisadaptiveoptimizers,chang2025convergencemuon,sato2025convergenceboundcriticalbatch,shen2025convergenceanalysismuon,li2025noteconvergencemuon} have been made to theoretically analyze the convergence properties of the Muon optimizer; the existing results are summarized in Table \ref{table:1}. These analyses aim to clarify the theoretical characteristics of Muon by establishing convergence guarantees or upper bounds under specific assumptions. However, each of these analyses is subject to limitations stemming from the strength of the imposed assumptions or the choice of evaluation criteria, and thus it does not necessarily fully capture the actual update structure of Muon.

\begin{table}[H]
\centering
\caption{Comparison of the existing results (1)--(7) with our results \textbf{(R1)}--\textbf{(R5)} for Muon. The existing results are (1) \citep[Theorem 4.6]{tang2025convergenceanalysisadaptiveoptimizers}, (2) \citep[Theorem 3.4]{chang2025convergencemuon}, (3) \citep[Theorem 3.1]{sato2025convergenceboundcriticalbatch}, (4) \citep[Theorem 4.1]{shen2025convergenceanalysismuon}, (5) \citep[Theorem 2.1]{li2025noteconvergencemuon}, (6) \citep[Theorem 5.5]{pethick2025training}, and (7) \citep[Theorem 4.9]{pethick2025generalized}, 
where Result (2) assumes the PL condition and Result (7) assumes $(L_0, L_1)$--smoothness condition. $(\bm{W}_t)$ is the sequence generated by Muon and $\overline{\bm{W}}_T$ is the mean of $(\bm{W}_t)_{t=0}^{T-1}$. $\|\cdot \|_\mathrm{F}$ is for the Frobenius norm and $\|\cdot\|_*$ is the dual norm of $\|\cdot \|_2$. $\mathbb{E}$ is the total expectation. Let $\eta > 0$, $b > 0$, $\delta > 1$, and $T$ be the number of steps. $O$ stands for Landau's symbol, i.e., $y_T = O(x_T)$ implies that, there exist $C > 0$ and $t_0 \in \mathbb{N}$ such that, for all $T \geq t_0$, $y_T \leq C x_T$, while $y_T = \mathcal{O}(x_T)$ implies that, for a fixed $T$, $y_T \leq C x_T$. The notation $O(x (T, \eta, b)) \hookrightarrow \mathcal{O}(x (T) )$ implies that, when $\eta$ and $b$ are set by using the fixed $T$, the convergence rate changes from $O(x (T, \eta, b))$ to $\mathcal{O}(x (T) )$.}
\label{table:1}
\footnotesize
\renewcommand{\arraystretch}{1.2}
\begin{tabular}{llcccc}
\toprule

Indicator & Learning Rate & Batch Size & Momentum Parameter & Convergence Rate\\
\midrule
(1) $\frac{1}{T}\sum_{t=0}^{T-1}\mathbb{E}[\|\nabla f(\bm{W}_t)\|_{\mathrm{F}}]$
& $\eta_t = \eta$
& $b_t = b = 1$
& $1-\beta = \mathcal{O}(\frac{1}{\sqrt{T}})$
& $\mathcal{O}(\frac{1}{T^{1/4}})$ \\
\midrule
(2) $\mathbb{E}[f(\bm{W}_{T+1})]-f^{\star}$
& $\eta_t = t^{-2/3}$
& $b_t = b$
& $\beta_t = 1-t^{-2/3}$
& $O (\frac{(\log{T})^2}{T^{2/3}})$ \\
\midrule
(3) $\frac{1}{T}\sum_{t=0}^{T-1}\mathbb{E}[\|\nabla f(\bm{W}_t)\|_{\mathrm{F}}]$
& $\eta_t = \eta$ 
& $b_t = b$
& $\beta \in [0,1]$
& $\mathcal{O}(\frac{1}{T}+\sqrt{\frac{1-\beta}{b}}+n)$ \\
\midrule
(4) $\frac{1}{T}\sum_{t=0}^{T-1}\mathbb{E}[\|\nabla f(\bm{W}_t)\|_*]$
& $\eta = \mathcal{O}(\frac{1}{\sqrt{T}})$
& $b_t = b = 1$
& $1 - \beta = \min \{\mathcal{O}(\frac{1}{\sqrt{T}}),1\}$
& $\mathcal{O} (\frac{1}{T^{1/4}})$ \\
\midrule
(5) $\frac{1}{T}\sum_{t=0}^{T-1}\mathbb{E}[\|\nabla f(\bm{W}_t)\|_{\mathrm{F}}]$
& $\eta = \mathcal{O}(\frac{1}{\sqrt{T}})$ 
& $b_t = b = T$
& $\beta \in [0,1]$
& $\mathcal{O} (\frac{1}{\sqrt{T}} )$ \\
\midrule
(6) $\mathbb{E}[\|\nabla f(\overline{\bm{W}}_T)\|_*]$
& $\frac{1}{2T^{3/4}}<\eta<\frac{1}{T^{3/4}}$
& $b_t=b$
& $\beta_t =\frac{1}{\sqrt{t}}$
& $\mathcal{O}(\frac{1}{T^{1/4}})$ \\
\midrule
(7) $\mathbb{E}[\|\nabla f(\overline{\bm{W}}_T)\|_*]$
& $\eta = \mathcal{O}(\frac{1}{T^{3/4}})$
& $b_t=b$
& $\beta =\frac{1}{\sqrt{T}}$
& $\mathcal{O}(\frac{1}{T^{1/4}})$ \\
\midrule\midrule
\textbf{(R1)}
$\min_{t} \mathbb{E}[ \| \nabla f (\bm{W}_t) \|_{\mathrm{F}}]$
& $\eta_t$ ($\leq \eta$) : Constant;
& $b_t = b$
& $\beta \in [0,1]$
& $O(\frac{1}{T} + \eta + \frac{1}{\sqrt{b}})$    \\
& Cosine; Polynomial
& $\hookrightarrow b = \mathcal{O}(T)$
&
& $\hookrightarrow \mathcal{O} (\frac{1}{\sqrt{T}})$ \\
& $\hookrightarrow \eta = \mathcal{O}(\frac{1}{T})$
&
&
&  \\
\midrule
\textbf{(R2)}
$\min_{t} \mathbb{E}[ \| \nabla f (\bm{W}_t) \|_{\mathrm{F}}]$
& $\eta_t$ ($\leq \eta$) : Constant;
& $b_t = b$
& $\beta \in [0,1]$
& $O(\frac{1}{T} + \eta + \frac{1}{\sqrt{b}})$    \\
& Cosine; Polynomial
& $\hookrightarrow b = \mathcal{O}(T^2)$
&
& $\hookrightarrow \mathcal{O}(\frac{1}{T})$ \\
& $\hookrightarrow \eta = \mathcal{O}(\frac{1}{T})$
&
&
&  \\
\midrule
\textbf{(R3)}
$\min_{t} \mathbb{E}[ \| \nabla f (\bm{W}_t) \|_{\mathrm{F}}]$
& $\eta_t$ ($\leq \eta$) : Constant;
& $b_t = b \delta^t$
& $\beta \in [0,1]$
& $O(\frac{1}{T} + \eta)$    \\
& Cosine; Polynomial
& : Exponential
&
& $\hookrightarrow \mathcal{O}(\frac{1}{T})$ \\
& $\hookrightarrow \eta = \mathcal{O}(\frac{1}{T})$
& ($\delta > 1$)
&
&  \\
\midrule
\textbf{(R4)}
$\min_{t} \mathbb{E}[ \| \nabla f (\bm{W}_t) \|_{\mathrm{F}}]$
& $\eta_t = \frac{\eta}{\sqrt{t+1}}$
& $b_t = b$ 
& $\beta \in [0,1]$
& $O(\frac{\log T}{\sqrt{T}} + \frac{1}{\sqrt{b}})$    \\
& 
& $\hookrightarrow b = \mathcal{O}(T)$
&
& $\hookrightarrow \mathcal{O}(\frac{\log T}{\sqrt{T}})$ \\
\midrule
\textbf{(R5)}
$\min_{t} \mathbb{E}[ \| \nabla f (\bm{W}_t) \|_{\mathrm{F}}]$
& $\eta_t = \frac{\eta}{\sqrt{t+1}}$
& $b_t = b \delta^t$
& $\beta \in [0,1]$
& $O(\frac{\log T}{\sqrt{T}} )$    \\
\bottomrule
\end{tabular}
\end{table}

Upon examining the results summarized in Table \ref{table:1}, we observe that Result (1) \citep[Theorem 4.6]{tang2025convergenceanalysisadaptiveoptimizers} provides the convergence guarantee of Muon; however, its convergence rate would be inferior to that obtained in the present work. Result (2) \citep[Theorem 3.4]{chang2025convergencemuon} can be interpreted as achieving a good convergence rate $O (\frac{(\log{T})^2}{T^{2/3}})$; however, it relies on the strong Polyak--\L ojasiewicz (PL) condition, which significantly restricts its range of applicability. In the case of (3) \citep[Theorem 3.1]{sato2025convergenceboundcriticalbatch}, a non-negligible variable $n$, where a parameter of a DNN is in $\mathbb{R}^{m \times n}$, remain in the analysis, and therefore a complete convergence result in a strict sense is not established. Result (4) \citep[Theorem 4.1]{shen2025convergenceanalysismuon}, similar to (1), provides a certain convergence guarantee, but its convergence rate would be again inferior to that achieved in this work. Result (5) \citep[Theorem 2.1]{li2025noteconvergencemuon} indicates that setting the number of steps $T$ as the batch size $b$ ensures a better $\mathcal {O}(\frac{1}{\sqrt{T}})$ convergence rate of Muon than an $\mathcal {O}(\frac{1}{T^{1/4}})$ convergence rate in Results (1) and (4). 
Result (6) \citep[Theorem 5.5]{pethick2025training} and
Result (7) \citep[Theorem 4.9]{pethick2025generalized} considered algorithms similar to Muon and showed that the algorithms have an $\mathcal {O}(\frac{1}{T^{1/4}})$ convergence rate.

Taken together, these observations indicate that the existing convergence analyses of the Muon optimizer remain insufficient in light of its demonstrated practical effectiveness, and that a more refined theoretical understanding is required. In particular, establishing clear convergence guarantees under more general assumptions constitute important challenges for the theoretical characterization of Muon. To address the issue, this work revisits the convergence behavior of the Muon optimizer and aims to establish more accurate and improved convergence guarantees than those provided by existing analyses.

\subsection{Contribution}
The main contributions of this paper are summarized as follows.

\textbf{Upper bound of the total expectation of the full gradient generated by the Muon optimizer:} We first present a convergence analysis of Muon (Theorem \ref{thm:1}) under standard conditions (Assumption \ref{assum:1}). The analysis indicates that an upper bound of $\mathbb{E}[\|\nabla f (\bm{W}_t)\|_{\mathrm{F}}]$ generated by Muon without Nesterov (resp. with Nesterov; see Algorithm \ref{muon}) consists of five (resp. six) quantities depending on learning rate $\eta_t$, batch size $b_t$, and momentum parameter $\beta$. The key ideas of the proof of Theorem \ref{thm:1} are based on simplified techniques, such as the descent lemma (\eqref{descent_lemma} and \eqref{ineq:1} in Section \ref{proof_Theorem_1}) that is satisfied for a smooth function (Assumption \ref{assum:1}(i)) and the orthogonalization structure (Step 9 in Algorithm \ref{muon} and \eqref{ineq:2} and \eqref{ineq:3} in Section \ref{proof_Theorem_1}). 

\textbf{Improved convergence rate for the Muon optimizer in the sense of the notation $\mathcal{O}$:} We evaluate the upper bound in the sense of the notation $\mathcal{O}$, i.e., the case where the number of steps $T$ needed to train the DNN is fixed, by using practical learning rates and batch sizes. For example, when using a learning rate $\eta_t$ $(\leq \eta)$ defined by a constant, cosine-annealing, or polynomial decay learning rate and a constant batch size $b_t = b$, Muon has the upper bound $O(\frac{1}{T} + \eta + \frac{1}{\sqrt{b}})$, as shown in (\textbf{R1}) and (\textbf{R2}) in Table \ref{table:1} (Corollary \ref{cor:1}(i), (iii), and (v)). Hence, using $\eta = \mathcal{O}(\frac{1}{T})$ and $b = \mathcal{O}(T)$ leads to an $\mathcal{O}(\frac{1}{\sqrt{T}})$ convergence rate, while using $\eta = \mathcal{O}(\frac{1}{T})$ and $b = \mathcal{O}(T^2)$ leads to an $\mathcal{O}(\frac{1}{T})$ convergence rate. This implies that the larger $b$ is, the faster the convergence becomes. Such a trend can be seen in SGD and its variants \citep{l.2018dont,umeda2025increasing,oowada2025faster}. Motivated by this trend, we also consider using an increasing batch size, such as an exponentially growing batch size $b_t = b \delta^t$ (e.g., $\delta = 2$ where the batch size is doubled every step (or epoch)). As a result, Muon has the upper bound $O(\frac{1}{T} + \eta)$, as seen in (\textbf{R3}) in Table \ref{table:1} (Corollary \ref{cor:1}(ii), (iv), and (vi)). Hence, using $\eta = \mathcal{O}(\frac{1}{T})$ leads to an $\mathcal{O}(\frac{1}{T})$ convergence rate. The convergence rate $\mathcal{O}(\frac{1}{T})$ of Muon shown in (\textbf{R2}) and (\textbf{R3}) is an improvement on the existing $\mathcal{O}(\frac{1}{\sqrt{T}})$ rate in Result (5) \citep[Theorem 2.1]{li2025noteconvergencemuon} , Result (6) \citep[Theorem 5.5]{pethick2025training}, and Result (7) \citep[Theorem 4.9]{pethick2025generalized}. 

\textbf{Convergence guarantee of the Muon optimizer in the sense of the notation $O$:} Result (2) \citep[Theorem 3.4]{chang2025convergencemuon} showed that, under the PL condition, Muon with a diminishing learning rate has an $O(\frac{(\log T)^2}{T^{2/3}})$ rate of convergence. Hence, we would like to verify whether Muon with a diminishing learning rate $\eta_t$ converges without assuming the PL condition. Using a diminishing learning rate $\eta_t = \frac{\eta}{\sqrt{t+1}}$ and a constant batch size $b_t = b$ ensures that the upper bound is $O(\frac{\log T}{\sqrt{T}} + \frac{1}{\sqrt{b}})$, as shown in (\textbf{R4}) in Table \ref{table:1} (Corollary \ref{cor:1}(vii)). This result, together with using an increasing batch size (see (\textbf{R3})), leads to the finding that using an exponentially growing batch size achieves an $O(\frac{\log T}{\sqrt{T}})$ convergence rate, as shown in (\textbf{R5}) in Table \ref{table:1} (Corollary \ref{cor:1}(viii)). 

\section{Preliminaries}
Here, we describe the notation and state some definitions. Let $\mathbb{N}$ be the set of natural numbers. Let $[n] \coloneqq \{1,2,\cdots,n\}$ and $[0:n] \coloneqq \{0,1,\cdots,n\}$ for $n\in \mathbb{N}$. Let $\mathbb{R}^{m \times n}$ be the set of $m \times n$ matrices with inner product $\bm{W}_1 \bullet \bm{W}_2$ ($\bm{W}_1, \bm{W}_2 \in \mathbb{R}^{m \times n}$) and the Frobenius norm $\|\bm{W}\|_{\mathrm{F}} \coloneqq \sqrt{\bm{W} \bullet \bm{W}}$. The dual norm of the spectral norm is denoted by $\|\cdot\|_*$. $\mathbb{E}_\xi [\cdot]$ and $\mathbb{V}_\xi [\cdot]$ denote the expectation and variance with respect to a random variable $\xi$, respectively. $\mathbb{E}[\cdot]$ denotes the total expectation. 

\subsection{Nonconvex optimization problem and assumptions}
Let $S \coloneqq \{(\bm{x}_1, \bm{y}_1), \cdots, (\bm{x}_i, \bm{y}_i), \cdots, (\bm{x}_N, \bm{y}_N) \}$ be the training set, where data point $\bm{x}_i$ is associated with label $\bm{y}_i$, and let $\bm{W} \in \mathbb{R}^{m \times n}$ be a parameter of a DNN. Let $f_i \colon \mathbb{R}^{m \times n} \to \mathbb{R}$ be the loss function corresponding to the $i$-th labeled training data $(\bm{x}_i, \bm{y}_i)$. Then, we would like to minimize the empirical risk defined for all $\bm{W} \in \mathbb{R}^{m \times n}$ by $f (\bm{W}) \coloneqq \frac{1}{N} \sum_{i=1}^N f_i (\bm{W})$. Since $f_i$ is nonconvex, the problem of minimizing the above $f$ is a nonconvex optimization problem. We assume that $f_i$ $(i\in [n])$ is differentiable. This paper thus considers finding a stationary point of $f$ defined by $\bm{W}^\star$ with $\nabla f (\bm{W}^\star) = \bm{O}$.

We make the following standard assumptions to use an optimizer to minimize $f$ (see, e.g., \citep[Section 2.1]{umeda2025increasing} for justification of Assumption \ref{assum:1}). 

\begin{assum}\label{assum:1}
\item[{\em (i)}] {\em [Smoothness of loss function]} $f_i \colon \mathbb{R}^{m \times n} \to \mathbb{R}$ $(i\in [N])$ is $L_i$-smooth, i.e., for all $\bm{W}_1, \bm{W}_2 \in \mathbb{R}^{m \times n}$, $\| \nabla f_i (\bm{W}_1) - \nabla f_i (\bm{W}_2) \|_{\mathrm{F}} \leq L_i \| \bm{W}_1 - \bm{W}_2 \|_{\mathrm{F}}$. $f_i (\bm{W})$. $f_i^\star \coloneqq \inf_{\bm{W} \in \mathbb{R}^{m \times n}} f_i (\bm{W})$ $(i\in [N])$ is finite. 
\item[{\em (ii)}] Let $\xi$ be a random variable that is independent of $\bm{W} \in \mathbb{R}^{m \times n}$. $\nabla f_\xi \colon \mathbb{R}^{m \times n} \to \mathbb{R}^{m \times n}$ is the stochastic gradient of $f$ such that 
   \begin{enumerate}
 \item[{\em (ii-1)}] {\em [Unbiasedness of stochastic gradient]} for all $\bm{W} \in \mathbb{R}^{m \times n}$, $\mathbb{E}_{\xi} [\nabla f_\xi (\bm{W}) ] = \nabla f(\bm{W})$;
 \item[{\em (ii-2)}] {\em [Boundedness of variance of stochastic gradient]} There exists $\sigma \geq 0$ such that, for all $\bm{W} \in \mathbb{R}^{m \times n}$, $\mathbb{V}_\xi [\nabla f_\xi (\bm{W}) ] \leq \sigma^2$.
   \end{enumerate}
\end{assum}

Assumption \ref{assum:1}(i) implies that the following descent lemma \citep[Lemma 5.7]{doi:10.1137/1.9781611974997} holds for $f$: for all $\bm{W}_1, \bm{W}_2 \in \mathbb{R}^{m \times n}$,
\begin{align}\label{descent_lemma}
    f(\bm{W}_1) \leq f(\bm{W}_2) + \nabla f(\bm{W}_2) \bullet (\bm{W}_1 - \bm{W}_2) + \frac{L}{2} \|\bm{W}_1 - \bm{W}_2\|_{\mathrm{F}}^2,
\end{align}
where $L \coloneqq \frac{1}{N} \sum_{i=1}^N L_i$. Moreover, Assumption \ref{assum:1}(i) guarantees that $f^\star \coloneqq \inf_{\bm{W} \in \mathbb{R}^{m \times n}} f(\bm{W}) < + \infty$ holds.

\subsection{Muon optimizer}
\begin{algorithm}[H]
 \caption{Muon}\label{muon}
    \begin{algorithmic}[1]
    \REQUIRE 
    $\eta_t > 0$,
    $b_t > 0$, 
    $\beta \in [0, 1)$, 
    $\bm{M}_{-1} \coloneqq \bm{O}$, $\bm{W}_0 \in \mathbb{R}^{m \times n}$,
    $T \in \mathbb{N}$.
    \FOR{$t = 0$ \TO $T - 1$}
        \STATE $\nabla f_{B_t} (\bm{W}_t) \coloneqq \frac{1}{b_t} \sum_{i=1}^{b_t} \nabla f_{\xi_{t,i}} (\bm{W}_t)$
        \STATE $\bm{M}_t \coloneqq \beta \bm{M}_{t-1} + (1 - \beta) \nabla f_{B_t} (\bm{W}_t)$
        \IF{(Nesterov = True)}
            \STATE $\bm{C}_t \coloneqq \beta \bm{M}_t + (1 - \beta) \nabla f_{B_t}(\bm{W}_t)$
        \ELSE
            \STATE $\bm{C}_t \coloneqq \bm{M}_t$
        \ENDIF
        \STATE $\bm{O}_t \coloneqq \underset{\bm{O}^\top \bm{O} = \bm{I}_n} {\operatorname{argmin}} \| \bm{O} - \bm{C}_t \|_{\mathrm{F}}$
            \STATE $\bm{W}_{t+1} \coloneqq \bm{W}_t - \eta_t \bm{O}_t$
    \ENDFOR
    \RETURN $\bm{W}_T$
    \end{algorithmic}
\end{algorithm}

Algorithm \ref{muon} \citep{jordan2024muon} presents the most common variant of Muon, which incorporates Nesterov momentum (Step 5 in Algorithm \ref{muon}). The Muon optimizer with Nesterov momentum is often used in practice (e.g., see \citep{jordan2024muon,ai2025practicalefficiencymuonpretraining}). 

Let $b_t$ be the batch size at $t \in \mathbb{N}$ and let $\bm{\xi}_t \coloneqq (\xi_{t,1}, \cdots, \xi_{t,i}, \cdots, \xi_{t, b_t})^\top$ comprise $b_t$ independent and identically distributed variables and be independent of $\bm{W} \in \mathbb{R}^{m \times n}$. Assumption \ref{assum:1}(ii) implies that the mini-batch stochastic gradient $\nabla f_{B_t}$ defined as in Step 2 of Algorithm \ref{muon} has the following properties for all $\bm{W} \in \mathbb{R}^{m \times n}$ and all $t \in \mathbb{N}$ (see, e.g., \citep[Proposition A.1]{umeda2025increasing}): 
\begin{align}\label{assum_ii_2}
    \mathbb{E}_{\bm{\xi}_t} \left[ \nabla f_{B_t} (\bm{W})\right] = \nabla f (\bm{W}) \text{ and }
    \mathbb{V}_{\bm{\xi}_t} \left[ \nabla f_{B_t} (\bm{W})\right] \leq \frac{\sigma^2}{b_t}.
\end{align}

\section{Convergence Analysis of Muon optimizer}
\subsection{Main result}
We present the following convergence analysis of Algorithm \ref{muon}. The proof of Theorem \ref{thm:1} is given in Section \ref{proof_Theorem_1}.

\begin{thm}\label{thm:1}
Suppose that Assumption \ref{assum:1} holds. Then, the following hold.
    \begin{enumerate}
\item[{\em (i)}] Muon without Nesterov ($\bm{C}_t = \bm{M}_t$) satisfies that, for all $T \in \mathbb{N}$, 
        \begin{align*}
        \min_{t \in [0:T-1]} 
        \mathbb{E}[\|\nabla f(\bm{W}_t)\|_{\mathrm{F}}]
        &\le
        \frac{C_1}{\sum_{t=0}^{T-1} \eta_t} 
        + C_2 \frac{\sum_{t=0}^{T-1} \eta_t^2}{\sum_{t=0}^{T-1} \eta_t}
        + C_3 \frac{\sum_{t=0}^{T-1} \eta_t \beta^t}{\sum_{t=0}^{T-1} \eta_t}\\ 
        &\quad + C_4  
        \frac{\sum_{t=0}^{T-1} \eta_t \sum_{i=1}^{t}\beta^i \eta_{t-i}}{\sum_{t=0}^{T-1} \eta_t}
        + C_5 
        \frac{\sum_{t=0}^{T-1} \eta_t \sum_{i=0}^{t} \frac{\beta^i}{\sqrt{b_{t-i}}}}{\sum_{t=0}^{T-1} \eta_t},
        \end{align*}
 where 
        \begin{align*}
            &C_1 \coloneqq f(\bm{W}_0) - f^\star, \text{ } 
            C_2 \coloneqq \frac{nL}{2}, 
            \text{ }
            C_3 \coloneqq 2 \|\bm{M}_0 - \nabla f(\bm{W}_0)\|_{\mathrm{F}} \sqrt{n},\\
            &C_4 \coloneqq 2 n L, \text{ and }
            C_5 \coloneqq 2 (1-\beta)\sigma \sqrt{n}.
        \end{align*}
 \item[{\em (ii)}] Muon with Nesterov ($\bm{C}_t = \beta \bm{M}_t + (1 - \beta) \nabla f_{B_t} (\bm{W}_t)$) satisfies that, for all $T \in \mathbb{N}$,
        \begin{align*}
        \min_{t \in [0:T-1]} \mathbb{E}[\|\nabla f(\bm{W}_t)\|]
        &\le 
        \frac{C_1}{\sum_{t=0}^{T-1} \eta_t} 
        + C_2 \frac{\sum_{t=0}^{T-1} \eta_t^2}{\sum_{t=0}^{T-1} \eta_t}
        + C_3 \frac{\sum_{t=0}^{T-1} \eta_t \beta^{t+1}}{\sum_{t=0}^{T-1} \eta_t}\\
        &\quad 
        + C_4 \frac{\sum_{t=0}^{T-1} \eta_t \sum_{i=1}^{t} \beta^{i+1} \eta_{t-i}}{\sum_{t=0}^{T-1} \eta_t}
        + C_5 \frac{\sum_{t=0}^{T-1} \eta_t \sum_{i=0}^{t} \frac{\beta^{i+1}}{\sqrt{b_{t-i}}}}{\sum_{t=0}^{T-1} \eta_t}\\
        &\quad 
        + C_6 \frac{\sum_{t=0}^{T-1} \frac{\eta_t}{\sqrt{b_t}}}{\sum_{t=0}^{T-1} \eta_t},
        \end{align*}
 where $C_i$ ($i\in [5]$) is defined as in Theorem \ref{thm:1}(i) and $C_6 \coloneqq 2 (1-\beta)\sigma \sqrt{n}$.
    \end{enumerate}
\end{thm}

\subsection{Concrete convergence rate}
Theorem \ref{thm:1} indicates that, if a learning rate $\eta_t$ and a batch size $b_t$ can be chosen such that, when $T$ is large enough, 
\begin{align}\label{conditions_1}
\begin{split}
&\frac{1}{\sum_{t=0}^{T-1} \eta_t}, \text{ }
\frac{\sum_{t=0}^{T-1} \eta_t^2}{\sum_{t=0}^{T-1} \eta_t}, \text{ }
\frac{\sum_{t=0}^{T-1} \eta_t \beta^{t}}{\sum_{t=0}^{T-1} \eta_t}, 
\frac{\sum_{t=0}^{T-1} \eta_t \sum_{i=1}^{t} \beta^{i} \eta_{t-i}}{\sum_{t=0}^{T-1} \eta_t}, \\
&\frac{\sum_{t=0}^{T-1} \eta_t \sum_{i=0}^{t} \frac{\beta^i}{\sqrt{b_{t-i}}}}{\sum_{t=0}^{T-1} \eta_t}, \text{ and }
\frac{\sum_{t=0}^{T-1} \frac{\eta_t}{\sqrt{b_t}}}{\sum_{t=0}^{T-1} \eta_t} \text{ are small enough,}
\end{split}
\end{align}
then the Muon optimizer can approximate a stationary point of $f$ in the sense of the total expectation of the Frobenius norm.

Let $\eta > 0$, $b > 0$, $p > 0$, $\delta > 1$, and $T \in \mathbb{N}$. We focus on the following four practical learning rates: for all $t \in [T]$,
\begin{align}\label{lr}
\eta_t 
= 
\begin{dcases}
\eta &\text{ [Constant LR]} \\
\frac{\eta}{2} \left(1 + \cos \frac{t \pi}{T} \right) &\text{ [Cosine-annealing LR]}\\
\eta \left( 1 - \frac{t}{T}  \right)^p &\text{ [Polynomial decay LR]}\\
\frac{\eta}{\sqrt{t+1}} &\text{ [Diminishing LR]}.
\end{dcases}
\end{align}
and the following two practical batch sizes: for all $t \in [T]$,
\begin{align}\label{bs}
b_t = 
\begin{dcases}
b  &\text{ [Constant BS]} \\  
b \delta^t  &\text{ [Exponentially growing BS]}.
\end{dcases}
\end{align}
We may modify the cosine-annealing LR in \eqref{lr} updated each step to one (e.g., $\eta_t = \frac{\eta}{2}(1 + \cos \lfloor \frac{t}{K} \rfloor \frac{\pi}{E})$, where $K \coloneqq \lceil \frac{N}{b}\rceil$ and $E$ is the number of epochs) updated each epoch and the exponentially growing BS in \eqref{bs} updated each step to the one updated each epoch. This is because the number of steps per epoch is finite and the upper bound in Theorem \ref{thm:1} is independent of the modification. 

From Theorem \ref{thm:1}, it is sufficient to check whether $\eta_t$ and $b_t$ defined by \eqref{lr} and \eqref{bs} satisfy Condition \eqref{conditions_1}. The following corollary indicates the convergence of the Muon optimizer with $\eta_t$ and $b_t$ satisfying Condition \eqref{conditions_1}. The proof of Corollary \ref{cor:1} is based on elementary operations (The detailed proof is given in Appendix \ref{proof_cor_1}).

\begin{cor}\label{cor:1}
Let $\eta > 0$, $b > 0$, $\beta \in [0,1)$, $p > 0$, and $\delta > 1$. Under the assumptions in Theorem \ref{thm:1}, the Muon optimizer with $\eta_t$ and $b_t$ defined as in \eqref{lr} and \eqref{bs} has the following properties: 
\begin{enumerate}
 \item[{\em (i)}] When using a constant LR $\eta_t = \eta$ and constant BS $b_t = b$, for all $T \in \mathbb{N}$, 
    \begin{align*}
    \min_{t \in [0:T-1]} \mathbb{E}[\|\nabla f(\bm{W}_t)\|]
    \leq 
    \begin{dcases}
    \frac{D_1}{T} + D_2 \eta + \frac{D_3}{\sqrt{b}} 
    = O \left( \frac{1}{T} + \eta + \frac{1}{\sqrt{b}} \right) \text{ (without Nesterov)}, \\
    \frac{D_1}{T} + D_2 \eta + \frac{D_4}{\sqrt{b}}
    = O \left( \frac{1}{T} + \eta + \frac{1}{\sqrt{b}} \right) \text{ (with Nesterov)},
    \end{dcases}
    \end{align*}
where $D_1 \coloneqq \eta^{-1} C_1 + (1 - \beta )^{-1} C_3$, $D_2 \coloneqq C_2 + \frac{C_4}{1 - \beta}$, $D_3 \coloneqq \frac{C_5}{1 - \beta}$, and $D_4 \coloneqq \frac{C_5 + C_6}{1 - \beta}$. In particular, for a given $T$, 
    \begin{align*}
    \min_{t \in [0:T-1]} \mathbb{E}[\|\nabla f(\bm{W}_t)\|]
    =
    \begin{dcases}
      \mathcal{O} \left( \frac{1}{\sqrt{T}} \right) &\text{ } \left(\eta = \mathcal{O} \left(\frac{1}{\sqrt{T}} \right) \land b = \mathcal{O}(T) \right), \\
      \mathcal{O} \left( \frac{1}{\sqrt{T}} \right) &\text{ } \left(\eta = \mathcal{O} \left(\frac{1}{T} \right) \land b = \mathcal{O}(T) \right), \\
      \mathcal{O} \left( \frac{1}{T} \right) &\text{ } \left(\eta = \mathcal{O} \left(\frac{1}{T} \right) \land b = \mathcal{O}(T^2) \right).
    \end{dcases}
    \end{align*}
\item[{\em (ii)}] When using a constant LR $\eta_t = \eta$ and an exponentially growing BS $b_t = b \delta^t$, for all $T \in \mathbb{N}$, 
    \begin{align*}
    \min_{t \in [0:T-1]} \mathbb{E}[\|\nabla f(\bm{W}_t)\|]
    \leq 
    \begin{dcases}
    \frac{D_1}{T} + D_2 \eta + \frac{\sqrt{\delta} D_3}{(\sqrt{\delta} - 1)\sqrt{b} T} 
    = O \left( \frac{1}{T} + \eta \right) \text{ (without Nesterov)}, \\
    \frac{D_1}{T} + D_2 \eta + \frac{\sqrt{\delta} D_4}{(\sqrt{\delta} - 1)\sqrt{b} T}
    = O \left( \frac{1}{T} + \eta \right) \text{ (with Nesterov)} .
    \end{dcases}
    \end{align*}
 In particular, for a given $T$, 
    \begin{align*}
    \min_{t \in [0:T-1]} \mathbb{E}[\|\nabla f(\bm{W}_t)\|]
    =
    \begin{dcases}
      \mathcal{O} \left( \frac{1}{\sqrt{T}} \right) &\text{ } \left(\eta = \mathcal{O} \left(\frac{1}{\sqrt{T}} \right)  \right), \\
      \mathcal{O} \left( \frac{1}{T} \right) &\text{ } \left(\eta = \mathcal{O} \left(\frac{1}{T} \right) \right).
    \end{dcases}
    \end{align*}

 \item[{\em (iii)}] When using a cosine-annealing LR $\eta_t = \frac{\eta}{2}(1 + \cos \frac{t \pi}{T})$ and a constant BS $b_t = b$, for all $T \in \mathbb{N}$,
    \begin{align*}
    \min_{t \in [0:T-1]} \mathbb{E}[\|\nabla f(\bm{W}_t)\|]
    \leq 
    \begin{dcases}
    \frac{D_1'}{T} + D_2' \eta + \frac{2 D_3}{\sqrt{b}} 
    = O \left( \frac{1}{T} + \eta + \frac{1}{\sqrt{b}} \right) \text{ (without Nesterov)}, \\
    \frac{D_1'}{T} + D_2' \eta + \frac{2 D_4}{\sqrt{b}}
    = O \left( \frac{1}{T} + \eta + \frac{1}{\sqrt{b}} \right) \text{ (with Nesterov)},
    \end{dcases}
    \end{align*}
 where $D_1' \coloneqq 2 \eta^{-1} C_1 + \eta C_2 + 2(1 - \beta )^{-1} C_3$ and $D_2' \coloneqq \frac{3 C_2}{4} + \frac{2 C_4}{1-\beta}$. In particular, for a given $T$, 
    \begin{align*}
    \min_{t \in [0:T-1]} \mathbb{E}[\|\nabla f(\bm{W}_t)\|]
    =
    \begin{dcases}
      \mathcal{O} \left( \frac{1}{\sqrt{T}} \right) &\text{ } \left(\eta = \mathcal{O} \left(\frac{1}{\sqrt{T}} \right) \land b = \mathcal{O}(T) \right), \\
      \mathcal{O} \left( \frac{1}{\sqrt{T}} \right) &\text{ } \left(\eta = \mathcal{O} \left(\frac{1}{T} \right) \land b = \mathcal{O}(T) \right), \\
      \mathcal{O} \left( \frac{1}{T} \right) &\text{ } \left(\eta = \mathcal{O} \left(\frac{1}{T} \right) \land b = \mathcal{O}(T^2) \right).
    \end{dcases}
    \end{align*}
 \item[{\em (iv)}] When using a cosine-annealing LR $\eta_t = \frac{\eta}{2}(1 + \cos \frac{t \pi}{T})$ and an exponentially growing BS $b_t = b \delta^t$, for all $T \in \mathbb{N}$,
    \begin{align*}
    \min_{t \in [0:T-1]} \mathbb{E}[\|\nabla f(\bm{W}_t)\|]
    \leq 
    \begin{dcases}
    \frac{D_1'}{T} + D_2' \eta + \frac{2 \sqrt{\delta} D_3}{(\sqrt{\delta} - 1)\sqrt{b} T} 
    = O \left( \frac{1}{T} + \eta \right) \text{ (without Nesterov)}, \\
    \frac{D_1'}{T} + D_2' \eta + \frac{2 \sqrt{\delta} D_4}{(\sqrt{\delta} - 1)\sqrt{b} T}
    = O \left( \frac{1}{T} + \eta \right) \text{ (with Nesterov)} .
    \end{dcases}
    \end{align*}
 In particular, for a given $T$, 
    \begin{align*}
    \min_{t \in [0:T-1]} \mathbb{E}[\|\nabla f(\bm{W}_t)\|]
    =
    \begin{dcases}
      \mathcal{O} \left( \frac{1}{\sqrt{T}} \right) &\text{ } \left(\eta = \mathcal{O} \left(\frac{1}{\sqrt{T}} \right) \right), \\
      \mathcal{O} \left( \frac{1}{T} \right) &\text{ } \left(\eta = \mathcal{O} \left(\frac{1}{T} \right)\right).
    \end{dcases}
    \end{align*}

 \item[{\em (v)}] When using a polynomial decay LR $\eta_t = \eta (1 - \frac{t}{T})^p$ and a constant BS $b_t = b$, for all $T \in \mathbb{N}$,
    \begin{align*}
    \min_{t \in [0:T-1]} \mathbb{E}[\|\nabla f(\bm{W}_t)\|]
    \leq (p+1)  
    \begin{dcases}
    \frac{D_1^{''}}{T} + D_2^{''} \eta + \frac{D_3}{\sqrt{b}}
    = O \left( \frac{1}{T} + \eta + \frac{1}{\sqrt{b}} \right) \text{ (without Nesterov)}, \\
    \frac{D_1^{''}}{T} + D_2^{''} \eta + \frac{D_4}{\sqrt{b}} 
    = O \left( \frac{1}{T} + \eta + \frac{1}{\sqrt{b}} \right) \text{ (with Nesterov)},
    \end{dcases}
    \end{align*}
 where $D_1^{''} \coloneqq \eta^{-1} C_1 + \eta C_2 + (1 - \beta )^{-1} C_3$ and $D_2^{''} \coloneqq \frac{C_2}{2p + 1} + \frac{C_4}{1-\beta}$. In particular, for a given $T$, 
    \begin{align*}
    \min_{t \in [0:T-1]} \mathbb{E}[\|\nabla f(\bm{W}_t)\|]
    =
    \begin{dcases}
      \mathcal{O} \left( \frac{1}{\sqrt{T}} \right) &\text{ } \left(\eta = \mathcal{O} \left(\frac{1}{\sqrt{T}} \right) \land b = \mathcal{O}(T) \right), \\
      \mathcal{O} \left( \frac{1}{\sqrt{T}} \right) &\text{ } \left(\eta = \mathcal{O} \left(\frac{1}{T} \right) \land b = \mathcal{O}(T) \right), \\
      \mathcal{O} \left( \frac{1}{T} \right) &\text{ } \left(\eta = \mathcal{O} \left(\frac{1}{T} \right) \land b = \mathcal{O}(T^2) \right).
    \end{dcases}
    \end{align*}
 \item[{\em (vi)}] When using a polynomial decay LR $\eta_t = \eta (1 - \frac{t}{T})^p$ and an exponentially growing BS $b_t = b \delta^t$, for all $T \in \mathbb{N}$,
    \begin{align*}
    \min_{t \in [0:T-1]} \mathbb{E}[\|\nabla f(\bm{W}_t)\|]
    \leq (p+1)  
    \begin{dcases}
    \frac{D_1^{''}}{T} + D_2^{''} \eta + \frac{\sqrt{\delta} D_3}{(\sqrt{\delta} -1)\sqrt{b}T}
    = O \left( \frac{1}{T} + \eta \right) \text{ (without Nesterov)}, \\
    \frac{D_1^{''}}{T} + D_2^{''} \eta + \frac{\sqrt{\delta} D_4}{(\sqrt{\delta} -1)\sqrt{b} T}
    = O \left( \frac{1}{T} + \eta \right) \text{ (with Nesterov)} .
    \end{dcases}
    \end{align*}
 In particular, for a given $T$, 
    \begin{align*}
    \min_{t \in [0:T-1]} \mathbb{E}[\|\nabla f(\bm{W}_t)\|]
    =
    \begin{dcases}
      \mathcal{O} \left( \frac{1}{\sqrt{T}} \right) &\text{ } \left(\eta = \mathcal{O} \left(\frac{1}{\sqrt{T}} \right) \right), \\
      \mathcal{O} \left( \frac{1}{T} \right) &\text{ } \left(\eta = \mathcal{O} \left(\frac{1}{T} \right)\right).
    \end{dcases}
    \end{align*}

\item[{\em (vii)}] When using a diminishing LR $\eta_t \coloneqq \frac{\eta}{\sqrt{t+1}}$ and a constant BS $b_t = b > 0$, for all $T \in \mathbb{N}$, 
    \begin{align*}
    \min_{t \in [0:T-1]} \mathbb{E}[\|\nabla f(\bm{W}_t)\|]
    \leq 
    \begin{dcases}
    \frac{D_1}{2\sqrt{T}} +  \frac{\eta D_2 \log T}{2\sqrt{T}} + \frac{D_3}{\sqrt{b}}
    = O \left( \frac{\log T}{\sqrt{T}} + \frac{1}{\sqrt{b}} \right) \text{ (without Nesterov)}, \\
    \frac{D_1}{2 \sqrt{T}} +  \frac{\eta D_2 \log T}{2 \sqrt{T}} + \frac{D_4}{\sqrt{b}} 
    = O \left( \frac{\log T}{\sqrt{T}} + \frac{1}{\sqrt{b}} \right) \text{ (with Nesterov)} .
    \end{dcases}
    \end{align*}
 In particular, for a given $T$, using $b = \mathcal{O} (T)$ implies that 
    \begin{align*}
    \min_{t \in [0:T-1]} \mathbb{E}[\|\nabla f(\bm{W}_t)\|]
    =
      \mathcal{O} \left( \frac{\log T}{\sqrt{T}} \right).
    \end{align*}
 \item[{\em (viii)}] When using a diminishing LR $\eta_t = \frac{\eta}{\sqrt{t+1}}$ and an exponentially growing BS $b_t = b \delta^t$, for all $T \in \mathbb{N}$,
    \begin{align*}
    \min_{t \in [0:T-1]} \mathbb{E}[\|\nabla f(\bm{W}_t)\|]
    \leq
    \begin{dcases}
    \frac{D_1}{2\sqrt{T}} + \frac{\eta D_2 \log T}{2\sqrt{T}} + \frac{\sqrt{\delta} D_3}{(\sqrt{\delta} -1)\sqrt{b T}} 
    = O \left( \frac{\log T}{\sqrt{T}} \right) \text{ (without Nesterov)}, \\
    \frac{D_1}{2 \sqrt{T}} +  \frac{\eta D_2 \log T}{2 \sqrt{T}} + \frac{\sqrt{\delta} D_4}{(\sqrt{\delta} -1)\sqrt{b T}}
    = O \left( \frac{\log T}{\sqrt{T}} \right) \text{ (with Nesterov)} .
    \end{dcases}
    \end{align*}
\end{enumerate}
\end{cor}

\subsection{Proof of Theorem \ref{thm:1}}\label{proof_Theorem_1}
We first prove the following lemma.

\begin{lemma}\label{lem:1}
Suppose that Assumption \ref{assum:1} holds. Then, the Muon optimizer satisfies that, for all $t \in \mathbb{N}$,
    \begin{align*}
    f(\bm{W}_t) - f(\bm{W}_{t+1})
    \geq
    \eta_t  \| \nabla f (\bm{W}_t)  \|_{\mathrm{F}}
    - 2 \sqrt{n} \eta_t \| \nabla f(\bm{W}_t) - \bm{C}_t \|_{\mathrm{F}}
    - \frac{n L \eta_t^2}{2}.
\end{align*}
\end{lemma}

\begin{proof}
From \eqref{descent_lemma}, we have 
\begin{align*}
f(\bm{W}_{t+1})
\le f(\bm{W}_t)
+ \nabla f(\bm{W}_t) \bullet (\bm{W}_{t+1} - \bm{W}_t)
+ \frac{L}{2}\|\bm{W}_{t+1} - \bm{W}_t\|_{\mathrm{F}}^2,
\end{align*}
which, together with $\bm{W}_{t+1} - \bm{W}_t = - \eta_t \bm{O}_t$ and $\|\bm{O}_t\|_{\mathrm{F}} = \sqrt{n}$, implies that
\begin{align}\label{ineq:1}
\begin{split}
f(\bm{W}_t) - f(\bm{W}_{t+1})
&\ge
- \nabla f(\bm{W}_t) \bullet (\bm{W}_{t+1} - \bm{W}_t)
- \frac{L}{2} \|\bm{W}_{t+1} - \bm{W}_t\|_{\mathrm{F}}^2 \\
&=
\eta_t \nabla f(\bm{W}_t) \bullet \bm{O}_t
- \frac{L}{2} \eta_t^2 \|\bm{O}_t\|_{\mathrm{F}}^2 \\
&=
\eta_t 
\underbrace{\bm{C}_t \bullet \bm{O}_t}_{A_t}
+ \eta_t \underbrace{(\nabla f(\bm{W}_t) - \bm{C}_t) \bullet \bm{O}_t}_{B_t}
- \frac{n L \eta_t^2}{2}.
\end{split}
\end{align}
Since we have
\begin{align*}
\bm{O}_t
&\coloneqq \arg\min
\left\{ \| \bm{O} - \bm{C}_t\|_{\mathrm{F}}^2 \colon \bm{O}^\top \bm{O} = \bm{I}_n \right\} \\
&= \arg\min_{O}
\left\{ \|\bm{O}\|_{\mathrm{F}}^2 
- 2 \bm{C}_t \bullet \bm{O} + \|\bm{C}_t\|_{\mathrm{F}}^2 \colon \bm{O}^\top \bm{O} = \bm{I}_n \right\},
\end{align*}
the definition of the dual norm $\|\cdot\|_*$ implies that 
\begin{align*} 
    A_t \coloneqq   
    \bm{C}_t \bullet \bm{O}_t 
    = \max \left\{
    \bm{C}_t \bullet \bm{O} \colon \bm{O}^\top \bm{O} = \bm{I}_n \right\}
    = \|\bm{C}_t \|_*. 
\end{align*}
The triangle inequality and the relation $\|\bm{X}\|_{\mathrm{F}} \leq \|\bm{X}\|_* \leq \sqrt{\mathrm{rank}(\bm{X})} \|\bm{X}\|_{\mathrm{F}}$ thus imply that
\begin{align}\label{ineq:2}
\begin{split}
   A_t 
   &= \|\bm{C}_t \|_*
   = \|\nabla f (\bm{W}_t) + (\bm{C}_t - \nabla f (\bm{W}_t))   \|_*
   \geq 
   \| \nabla f (\bm{W}_t)  \|_*
   - 
   \|\bm{C}_t - \nabla f (\bm{W}_t)\|_*\\
   &\geq
   \| \nabla f (\bm{W}_t)  \|_{\mathrm{F}}
   - \sqrt{n} \|\bm{C}_t - \nabla f (\bm{W}_t)\|_{\mathrm{F}}.
\end{split}
\end{align}
Moreover, the Cauchy--Schwartz inequality implies 
\begin{align}\label{ineq:3}
\begin{split}
    B_t \coloneqq (\nabla f(\bm{W}_t) - \bm{C}_t) \bullet \bm{O}_t
    \geq 
    - \| \nabla f(\bm{W}_t) - \bm{C}_t \|_{\mathrm{F}} \| \bm{O}_t \|_{\mathrm{F}}
    = - \sqrt{n} \| \nabla f(\bm{W}_t) - \bm{C}_t \|_{\mathrm{F}}.
\end{split}    
\end{align}
Form \eqref{ineq:1}, \eqref{ineq:2}, and \eqref{ineq:3}, 
\begin{align*}
    &f(\bm{W}_t) - f(\bm{W}_{t+1})\\
    &\geq
    \eta_t 
    \left\{ \| \nabla f (\bm{W}_t)  \|_{\mathrm{F}}
    - \sqrt{n} \|\bm{C}_t - \nabla f (\bm{W}_t)\|_{\mathrm{F}} \right\}
    - \sqrt{n} \eta_t \| \nabla f(\bm{W}_t) - \bm{C}_t \|_{\mathrm{F}}
    - \frac{n L \eta_t^2}{2}\\
    &= 
    \eta_t  \| \nabla f (\bm{W}_t)  \|_{\mathrm{F}}
    - 2 \sqrt{n} \eta_t \| \nabla f(\bm{W}_t) - \bm{C}_t \|_{\mathrm{F}}
    - \frac{n L \eta_t^2}{2},
\end{align*}
which completes the proof. 
\end{proof}

Next, we give the following lemma in order to evaluate $\| \nabla f(\bm{W}_t) - \bm{C}_t \|$ in Lemma \ref{lem:1}.

\begin{lemma}\label{lem:2}
Suppose that Assumption \ref{assum:1} holds and consider the Muon optimizer. Then, 
    \begin{enumerate}
\item[{\em (i)}] Muon without Nesterov ($\bm{C}_t = \bm{M}_t$) satisfies that, for all $t\in \mathbb{N}$,
         \begin{align*}
         \mathbb{E} \left[ \| \nabla f(\bm{W}_t) - \bm{M}_t \|_{\mathrm{F}} \right]
         \leq 
         \beta^t \|\bm{M}_0 - \nabla f(\bm{W}_0)\|_{\mathrm{F}} + L \sqrt{n} \sum_{i=1}^{t}\beta^i \eta_{t-i}
        +(1-\beta)\sigma \sum_{i=0}^{t} \frac{\beta^i}{\sqrt{b_{t-i}}}.
         \end{align*}
\item[{\em (ii)}] Muon with Nesterov ($\bm{C}_t = \beta\bm{M}_t + (1-\beta) \nabla f_{B_t} (\bm{W}_t)$) satisfies that, for all $t\in \mathbb{N}$,
         \begin{align*}
         &\mathbb{E} \left[ \| \nabla f(\bm{W}_t) - \bm{C}_t \|_{\mathrm{F}} \right]\\
         &\leq 
         \beta^{t+1} \|\bm{M}_0 - \nabla f(\bm{W}_0)\|_{\mathrm{F}}]
        + \beta L\sqrt{n} \sum_{i=1}^t \beta^i \eta_{t-i} 
        + \beta(1-\beta)\sigma \sum_{i=0}^t \frac{\beta^i}{\sqrt{b_{t-i}}}
        + \frac{(1-\beta)\sigma}{\sqrt{b_t}}.
         \end{align*}
    \end{enumerate}
\end{lemma}

\begin{proof}
The definition of $\bm{M}_t$ and the triangle inequality imply that
\begin{align}\label{ineq:4}
\begin{split}
\| \bm{M}_t - \nabla f(\bm{W}_t) \|_{\mathrm{F}}
&\leq 
\beta \| \bm{M}_{t-1} - \nabla f(\bm{W}_{t-1}) \|_{\mathrm{F}} 
+ \beta \| \nabla f(\bm{W}_{t-1}) -\nabla f(\bm{W}_t) \|_{\mathrm{F}} \\
&\quad +(1-\beta) \| \nabla f_{B_t}(\bm{W}_t)-\nabla f(\bm{W}_t)\|_{\mathrm{F}}.
\end{split}
\end{align}
From $\bm{W}_{t} - \bm{W}_{t-1} = - \eta_{t-1} \bm{O}_{t-1}$ and Assumption \ref{assum:1}(i), we have $\| \nabla f(\bm{W}_{t-1}) -\nabla f(\bm{W}_t) \|_{\mathrm{F}} \leq L \|\bm{W}_{t-1} - \bm{W}_t\|_{\mathrm{F}} = L \sqrt{n} \eta_{t-1}$. Moreover, \eqref{assum_ii_2} implies that
\begin{align*}
    \mathbb{E}_{\bm{\xi}_t}\left[ \| \nabla f_{B_t}(\bm{W}_t)-\nabla f(\bm{W}_t)\|_{\mathrm{F}} \right]
    \leq \sqrt{\mathbb{V}_{\bm{\xi}_t}[\nabla f_{B_t}(\bm{W}_t)]}
    \leq \frac{\sigma}{\sqrt{b_t}}.
\end{align*}
Accordingly, taking the total expectation in \eqref{ineq:4} implies that 
\begin{align*}
\mathbb{E}[\|\bm{M}_t - \nabla f(\bm{W}_t)\|_{\mathrm{F}}] 
\le
\beta\mathbb{E}[\|\bm{M}_{t-1} - \nabla f(\bm{W}_{t-1})\|_{\mathrm{F}}]
+ \beta L \sqrt{n} \eta_{t-1} 
+ \frac{(1-\beta)\sigma}{\sqrt{b_t}}. 
\end{align*}
Induction thus implies that 
\begin{align*}
\mathbb{E}[\|\bm{M}_t - \nabla f(\bm{W}_t)\|_{\mathrm{F}}] 
&\le
\beta^t \|\bm{M}_0 - \nabla f(\bm{W}_0)\|_{\mathrm{F}} + L \sqrt{n} \sum_{i=1}^{t}\beta^i \eta_{t-i}
+(1-\beta)\sigma \sum_{i=0}^{t} \frac{\beta^i}{\sqrt{b_{t-i}}}. 
\end{align*}

(i) From $\bm{C}_t = \bm{M}_t$, we have Lemma \ref{lem:2}(i). 

(ii) From $\bm{C}_t = \beta\bm{M}_t + (1-\beta) \nabla f_{B_t} (\bm{W}_t)$, 
\begin{align*}
\|\bm{C}_t - \nabla f(\bm{W}_t)\| 
&\leq
\beta \|\bm{M}_t - \nabla f(\bm{W}_t)\| + (1-\beta)\|\nabla f_{B_t}(\bm{W}_t) - \nabla f(\bm{W}_t)\|_{\mathrm{F}},
\end{align*}
which, together with Lemma \ref{lem:2}(i), implies that
\begin{align*}
&\mathbb{E}[ \|\bm{C}_t - \nabla f(\bm{W}_t)\|_{\mathrm{F}}]\\
&\leq
\beta 
\left\{ 
\beta^t \|\bm{M}_0 - \nabla f(\bm{W}_0)\|_{\mathrm{F}} + L \sqrt{n} \sum_{i=1}^{t}\beta^i \eta_{t-i}
+(1-\beta)\sigma \sum_{i=0}^{t} \frac{\beta^i}{\sqrt{b_{t-i}}}
\right\}
+ 
(1 - \beta) \frac{\sigma}{\sqrt{b_t}}\\
&=
\beta^{t+1} \|\bm{M}_0 - \nabla f(\bm{W}_0)\|_{\mathrm{F}}
+ \beta L\sqrt{n} \sum_{i=1}^t \beta^i \eta_{t-i} 
+ \beta(1-\beta)\sigma \sum_{i=0}^t \frac{\beta^i}{\sqrt{b_{t-i}}}
+ \frac{(1-\beta)\sigma}{\sqrt{b_t}},
\end{align*}
which completes the proof. 
\end{proof}

We are now in a position to prove Theorem \ref{thm:1}. 

\begin{proof}\textbf{of Theorem} \ref{thm:1}
    
(i) Lemma \ref{lem:1} and Lemma \ref{lem:2}(i) imply that, for all $t \in \mathbb{N}$,
\begin{align*}
&\eta_t \mathbb{E}[\|\nabla f(\bm{W}_t)\|_{\mathrm{F}}]\\ 
&\le
\mathbb{E} [ f(\bm{W}_t) - f(\bm{W}_{t+1}) ] +\frac{nL\eta_t^2}{2}\\
&\quad 
+ 2 \sqrt{n} \eta_t 
\left\{
\beta^t \|\bm{M}_0 - \nabla f(\bm{W}_0)\|_{\mathrm{F}} 
+ L \sqrt{n} \sum_{i=1}^{t}\beta^i \eta_{t-i}
+ (1-\beta)\sigma \sum_{i=0}^{t} \frac{\beta^i}{\sqrt{b_{t-i}}}
\right\}\\
&= 
\mathbb{E} [ f(\bm{W}_t) - f(\bm{W}_{t+1}) ] +\frac{nL\eta_t^2}{2}
+ 2 \|\bm{M}_0 - \nabla f(\bm{W}_0)\|_{\mathrm{F}} \sqrt{n} \eta_t \beta^t\\
&\quad 
+ 2 L n \eta_t \sum_{i=1}^{t}\beta^i \eta_{t-i}
+ 2 (1-\beta)\sigma \sqrt{n} \eta_t \sum_{i=0}^{t} \frac{\beta^i}{\sqrt{b_{t-i}}}.
\end{align*}
Let $T \in \mathbb{N}$. Summing the above inequality from $t=0$ to $t = T-1$ ensures that
\begin{align*}
\sum_{t=0}^{T-1} \eta_t \mathbb{E}[\|\nabla f(\bm{W}_t)\|_{\mathrm{F}}] 
&\leq
f(\bm{W}_0) - f^\star +\frac{nL}{2} \sum_{t=0}^{T-1} \eta_t^2
+ 2 \|\bm{M}_0 - \nabla f(\bm{W}_0)\|_{\mathrm{F}} \sqrt{n} \sum_{t=0}^{T-1} \eta_t \beta^t\\
&\quad 
+ 2 L n \sum_{t=0}^{T-1} \eta_t \sum_{i=1}^{t}\beta^i \eta_{t-i}
+ 2 (1-\beta)\sigma \sqrt{n} \sum_{t=0}^{T-1} \eta_t \sum_{i=0}^{t} \frac{\beta^i}{\sqrt{b_{t-i}}},
\end{align*}
which, together with $\min_{t \in [0:T-1]} \mathbb{E}[\|\nabla f(\bm{W}_t)\|_{\mathrm{F}}] \leq \sum_{t=0}^{T-1} \eta_t \mathbb{E}[\|\nabla f(\bm{W}_t)\|_{\mathrm{F}}]/\sum_{t=0}^{T-1} \eta_t$, implies that
\begin{align*}
\min_{t \in [0:T-1]} \mathbb{E}[\|\nabla f(\bm{W}_t)\|_{\mathrm{F}}]
&\leq
\frac{f(\bm{W}_0) - f^\star}{\sum_{t=0}^{T-1} \eta_t} +\frac{nL}{2} \frac{\sum_{t=0}^{T-1} \eta_t^2}{\sum_{t=0}^{T-1} \eta_t}
+ 2 \|\bm{M}_0 - \nabla f(\bm{W}_0)\|_{\mathrm{F}} \sqrt{n} \frac{\sum_{t=0}^{T-1} \eta_t \beta^t}{\sum_{t=0}^{T-1} \eta_t}\\
&\quad 
+ 2 L n 
\frac{\sum_{t=0}^{T-1} \eta_t \sum_{i=1}^{t}\beta^i \eta_{t-i}}{\sum_{t=0}^{T-1} \eta_t}
+ 2 (1-\beta)\sigma \sqrt{n} 
\frac{\sum_{t=0}^{T-1} \eta_t \sum_{i=0}^{t} \frac{\beta^i}{\sqrt{b_{t-i}}}}{\sum_{t=0}^{T-1} \eta_t}.
\end{align*}

(ii) Lemma \ref{lem:1} and Lemma \ref{lem:2}(ii) imply that, for all $t \in \mathbb{N}$,
\begin{align*}
&\eta_t \mathbb{E}[\|\nabla f(\bm{W}_t)\|_{\mathrm{F}}]\\ 
&\le
\mathbb{E} [ f(\bm{W}_t) - f(\bm{W}_{t+1}) ] +\frac{nL\eta_t^2}{2}\\
&\quad 
+ 2 \sqrt{n} \eta_t 
\left\{
\beta^{t+1} \|\bm{M}_0 - \nabla f(\bm{W}_0)\|_{\mathrm{F}}]
        + \beta L\sqrt{n} \sum_{i=1}^t \beta^i \eta_{t-i} 
        + \beta(1-\beta)\sigma \sum_{i=0}^t \frac{\beta^i}{\sqrt{b_{t-i}}}
        + \frac{(1-\beta)\sigma}{\sqrt{b_t}}
\right\}\\
&= 
\mathbb{E} [ f(\bm{W}_t) - f(\bm{W}_{t+1}) ] +\frac{nL\eta_t^2}{2}
+ 2 \|\bm{M}_0 - \nabla f(\bm{W}_0)\|_{\mathrm{F}} \sqrt{n} \eta_t \beta^{t+1}\\
&\quad 
+ 2 \beta L n \eta_t \sum_{i=1}^{t}\beta^i \eta_{t-i}
+ 2 \beta (1-\beta)\sigma \sqrt{n} \eta_t \sum_{i=0}^{t} \frac{\beta^i}{\sqrt{b_{t-i}}}
+  \frac{2 \sqrt{n} (1-\beta)\sigma \eta_t}{\sqrt{b_t}}.
\end{align*}
A discussion similar to the one proving Theorem \ref{thm:1}(i) implies that Theorem \ref{thm:1}(ii) holds. This completes the proof.
\end{proof}

\section{Conclusion}
We presented a comprehensive convergence analysis of the Muon optimizer. By investigating combinations of Nesterov acceleration, four types of learning rates, and two batch-size configurations, we derived improved convergence rates under significantly weaker assumptions compared with existing studies. Our analysis highlights that the combination of a diminishing learning rate and an exponentially growing batch size yields the most favorable results in terms of both stability and convergence rate. Furthermore, we demonstrated that by appropriately coupling the learning rate and batch size as functions of the total number of iterations $T$, superior convergence rates can also be achieved across other combinations. These findings provide a theoretical foundation for more efficient hyperparameter tuning in Muon-based optimization.

\bibliographystyle{tmlr}
\bibliography{Muon_convergence}
\clearpage
\appendix

\section{Proof of Corollary \ref{cor:1}}\label{proof_cor_1}

\begin{proof}
(i) When using $\eta_t \coloneqq \eta$ and $b_t \coloneqq b$, we have 
\begin{align*}
&\frac{1}{\sum_{t=0}^{T-1} \eta_t} = \frac{1}{\eta T}, \text{ }
\frac{\sum_{t=0}^{T-1} \eta_t^2}{\sum_{t=0}^{T-1} \eta_t} = \eta, \text{ }
\frac{\sum_{t=0}^{T-1} \eta_t \beta^{t}}{\sum_{t=0}^{T-1} \eta_t}
= \frac{1}{T} \sum_{t=0}^{T-1} \beta^t
\leq \frac{1}{(1 - \beta)T},\\ 
&\frac{\sum_{t=0}^{T-1} \eta_t \sum_{i=1}^{t} \beta^{i} \eta_{t-i}}{\sum_{t=0}^{T-1} \eta_t}
= \frac{\eta}{T} \sum_{t=0}^{T-1} \sum_{i=1}^t \beta^i = \frac{\eta}{T} \sum_{t=0}^{T-1} \frac{\beta(1 - \beta^t)}{1 - \beta}
\leq \frac{\eta}{1-\beta}, \\
&\frac{\sum_{t=0}^{T-1} \eta_t \sum_{i=0}^{t} \frac{\beta^i}{\sqrt{b_{t-i}}}}{\sum_{t=0}^{T-1} \eta_t}
= 
\frac{1}{T \sqrt{b}} \sum_{t=0}^{T-1} \sum_{i=0}^{t} \beta^i
\leq 
\frac{1}{(1 - \beta) \sqrt{b}}, \text{ }
\frac{\sum_{t=0}^{T-1} \frac{\eta_t}{\sqrt{b_t}}}{\sum_{t=0}^{T-1} \eta_t}
= \frac{1}{\sqrt{b}}.
\end{align*}
Theorem \ref{thm:1} thus leads to Corollary \ref{cor:1}(i).

(ii) When using $\eta_t \coloneqq \eta$ and $b_t \coloneqq b \delta^t$, we have 
\begin{align*}
&\frac{\sum_{t=0}^{T-1} \eta_t \sum_{i=0}^{t} \frac{\beta^i}{\sqrt{b_{t-i}}}}{\sum_{t=0}^{T-1} \eta_t}
= 
\frac{1}{T \sqrt{b}} \sum_{t=0}^{T-1} \sum_{i=0}^{t} \frac{\beta^i}{\sqrt{\delta^{t-i}}}
\leq 
\frac{1}{(1-\beta)T \sqrt{b}} \sum_{t=0}^{T-1} 
\frac{1}{\sqrt{\delta^t}}
\leq 
\frac{\sqrt{\delta}}{(1-\beta)(\sqrt{\delta} - 1)T \sqrt{b}},\\ 
&\frac{\sum_{t=0}^{T-1} \frac{\eta_t}{\sqrt{b_t}}}{\sum_{t=0}^{T-1} \eta_t}
= 
\frac{1}{T \sqrt{b}} \sum_{t=0}^{T-1} \frac{1}{\sqrt{\delta^t}}
\leq 
\frac{\sqrt{\delta}}{(\sqrt{\delta} - 1)T \sqrt{b}}.
\end{align*}
Hence, Theorem \ref{thm:1} thus leads to Corollary \ref{cor:1}(ii).

(iii) Section A.3 in \citep{umeda2025increasing} ensures that $\eta_t = \frac{\eta}{2} (1 + \cos \frac{t \pi}{T})$ satisfies 
\begin{align*}
\sum_{t=0}^{T-1} \eta_t \geq \frac{\eta T}{2}
\text{ and }
\sum_{t=0}^{T-1} \eta_t^2 = \frac{3 \eta^2 T}{8} + \frac{\eta^2}{2}.
\end{align*}
Hence, 
\begin{align*}
\frac{1}{\sum_{t=0}^{T-1} \eta_t}
\leq 
\frac{2}{\eta T}
\text{ and }
\frac{\sum_{t=0}^{T-1} \eta_t^2}{\sum_{t=0}^{T-1} \eta_t}
\leq 
\frac{\eta}{T} + \frac{3 \eta}{4}.
\end{align*}
Using $b_t = b$ implies that 
\begin{align*}
&\frac{\sum_{t=0}^{T-1} \eta_t \beta^{t}}{\sum_{t=0}^{T-1} \eta_t}
\leq \frac{2}{T} \sum_{t=0}^{T-1} \beta^t
\leq \frac{2}{(1 - \beta)T},\\ 
&\frac{\sum_{t=0}^{T-1} \eta_t \sum_{i=1}^{t} \beta^{i} \eta_{t-i}}{\sum_{t=0}^{T-1} \eta_t}
\leq \frac{2 \eta}{T} \sum_{t=0}^{T-1} \sum_{i=1}^t \beta^i = \frac{2 \eta}{T} \sum_{t=0}^{T-1} \frac{\beta(1 - \beta^t)}{1 - \beta}
\leq \frac{2 \eta}{1-\beta}.
\end{align*}
Moreover, 
\begin{align*}
\frac{\sum_{t=0}^{T-1} \eta_t \sum_{i=0}^{t} \frac{\beta^i}{\sqrt{b_{t-i}}}}{\sum_{t=0}^{T-1} \eta_t}
\leq
\frac{2}{T \sqrt{b}} \sum_{t=0}^{T-1} \sum_{i=0}^{t} \beta^i
\leq 
\frac{2}{(1 - \beta) \sqrt{b}} \text{ and }
\frac{\sum_{t=0}^{T-1} \frac{\eta_t}{\sqrt{b_t}}}{\sum_{t=0}^{T-1} \eta_t}
= \frac{1}{\sqrt{b}}.
\end{align*}
Hence, Theorem \ref{thm:1} thus leads to Corollary \ref{cor:1}(iii).

(iv) Using the proof of Corollaries \ref{cor:1}(ii) and (iii) leads to Corollary \ref{cor:1}(iv).

(v) Section A.3 in \citep{umeda2025increasing} ensures that $\eta_t = \eta (1 - \frac{t}{T})^p$ satisfies 
\begin{align*}
\sum_{t=0}^{T-1} \eta_t \geq \frac{\eta T}{p + 1}  
\text{ and }
\sum_{t=0}^{T-1} \eta_t^2 \leq \frac{\eta^2 (2p + T + 1)}{2p + 1}.
\end{align*}
Hence, 
\begin{align*}
\frac{1}{\sum_{t=0}^{T-1} \eta_t}
\leq 
\frac{p+1}{\eta T}
\text{ and }
\frac{\sum_{t=0}^{T-1} \eta_t^2}{\sum_{t=0}^{T-1} \eta_t}
\leq 
\frac{(p+1)(2p+T+1)\eta}{(2p+1)T}.
\end{align*}
Using $b_t = b$ implies that 
\begin{align*}
&\frac{\sum_{t=0}^{T-1} \eta_t \beta^{t}}{\sum_{t=0}^{T-1} \eta_t}
\leq \frac{p+1}{T} \sum_{t=0}^{T-1} \beta^t
\leq \frac{p+1}{(1 - \beta)T},\\ 
&\frac{\sum_{t=0}^{T-1} \eta_t \sum_{i=1}^{t} \beta^{i} \eta_{t-i}}{\sum_{t=0}^{T-1} \eta_t}
\leq \frac{(p+1) \eta}{T} \sum_{t=0}^{T-1} \sum_{i=1}^t \beta^i = \frac{(p+1) \eta}{T} \sum_{t=0}^{T-1} \frac{\beta(1 - \beta^t)}{1 - \beta}
\leq \frac{(p+1) \eta}{1-\beta}.
\end{align*}
Moreover, 
\begin{align*}
\frac{\sum_{t=0}^{T-1} \eta_t \sum_{i=0}^{t} \frac{\beta^i}{\sqrt{b_{t-i}}}}{\sum_{t=0}^{T-1} \eta_t}
\leq
\frac{p+1}{T \sqrt{b}} \sum_{t=0}^{T-1} \sum_{i=0}^{t} \beta^i
\leq 
\frac{p+1}{(1 - \beta) \sqrt{b}} \text{ and }
\frac{\sum_{t=0}^{T-1} \frac{\eta_t}{\sqrt{b_t}}}{\sum_{t=0}^{T-1} \eta_t}
= \frac{1}{\sqrt{b}}.
\end{align*}
Hence, Theorem \ref{thm:1} thus leads to Corollary \ref{cor:1}(v).

(vi) The proof of Corollary \ref{cor:1}(vi) is similar to those of Corollaries \ref{cor:1}(ii) and (v).

(vii) From $\eta_t = \frac{\eta}{(t+1)^a}$, we have 
\begin{align*}
\sum_{t=0}^{T-1} \eta_t 
\geq \eta \int_0^T \frac{\mathrm{d}t}{(t+1)^a}
= 
\begin{dcases}
    \frac{\eta}{1 - a} \{ (T+1)^{1-a} - 1  \} \text{ } &(a \in (0,1)),\\
    \eta \log (T+1) \text{ } &(a =1)
\end{dcases}
\end{align*}
and 
\begin{align*}
\sum_{t=0}^{T-1} \eta_t^2 
\leq \eta^2
\left(1 + \int_0^{T-1} \frac{\mathrm{d}t}{(t+1)^{2a}} \right)
= 
\begin{dcases}
    \frac{\eta^2}{1 - 2a} T^{1 - 2a}  \text{ } &\left(a \in \left(0,\frac{1}{2} \right) \right),\\
    \eta^2 (1 + \log T) \text{ } &\left(a = \frac{1}{2} \right), \\
    \frac{2a \eta^2}{2a - 1} \text{ } &\left(a \in \left(\frac{1}{2}, 1 \right] \right).
\end{dcases}
\end{align*}
Accordingly, 
\begin{align*}
\frac{1}{\sum_{t=0}^{T-1} \eta_t}
\leq 
\begin{dcases}
    \frac{1 - a}{\eta (T^{1-a}-1)} \text{ } &(a \in (0,1)),\\
    \frac{1}{\eta \log T} \text{ } &(a =1)
\end{dcases}
\text{ and }
\frac{\sum_{t=0}^{T-1} \eta_t^2}{\sum_{t=0}^{T-1} \eta_t} 
\leq 
\begin{dcases}
\frac{(1-a)\eta T^{1-2a}}{(1 - 2a) (T^{1-a}-1)} \text{ } &\left( a \in \left(0, \frac{1}{2} \right)  \right), \\
\frac{\eta(1 + \log T)}{2 (\sqrt{T}-1)} \text{ } &\left( a = \frac{1}{2}  \right), \\
\frac{2 a (1-a) \eta}{(2a -1) (T^{1-a}-1)} \text{ } &\left( a \in \left(\frac{1}{2} ,1 \right)  \right), \\
\frac{2\eta}{\log T} \text{ } &\left( a = 1  \right).
\end{dcases}
\end{align*}
From $\sum_{t=0}^{T-1} \eta_t \beta^t \le \eta \sum_{t=0}^{T-1} \beta^t < \frac{\eta}{1-\beta}$, we have 
\begin{align*}
\frac{\sum_{t=0}^{T-1} \eta_t \beta^{t}}{\sum_{t=0}^{T-1} \eta_t}
\leq
\begin{dcases}
\frac{1 - a}{(1 - \beta) (T^{1-a}-1)} \text{ } &(a \in (0,1)),\\
\frac{1}{(1 - \beta) \log T} \text{ } &(a =1).
\end{dcases}
\end{align*}
Moreover, since we have 
\begin{align*}
\sum_{t=0}^{T-1}\eta_t \sum_{i=0}^{t}\beta^i\eta_{t-i}
\le
\sum_{k=0}^{T-1}\eta_k^2\sum_{t=k}^{T-1}\beta^{t-k} 
\leq 
\frac{1}{1-\beta}\sum_{k=0}^{T-1}\eta_k^2,
\end{align*}
we also have 
\begin{align*}
\frac{\sum_{t=0}^{T-1}\eta_t \sum_{i=0}^{t}\beta^i\eta_{t-i}}{\sum_{t=0}^{T-1}\eta_t}
\leq
\begin{dcases}
\frac{(1-a)\eta T^{1-2a}}{(1 - 2a)(1-\beta) (T^{1-a}-1)} \text{ } &\left( a \in \left(0, \frac{1}{2} \right)  \right), \\
\frac{\eta (1 + \log T)}{2 (1-\beta) (\sqrt{T} - 1)} \text{ } &\left( a = \frac{1}{2}  \right), \\
\frac{2 a (1-a) \eta}{(2a -1) (1-\beta) (T^{1-a}-1)} \text{ } &\left( a \in \left(\frac{1}{2} ,1 \right)  \right), \\
\frac{2 \eta}{(1-\beta) \log T} \text{ } &\left( a = 1  \right).
\end{dcases}
\end{align*}
From $b_t = b$, we have 
\begin{align*}
&\frac{\sum_{t=0}^{T-1} \eta_t \sum_{i=0}^{t} \frac{\beta^i}{\sqrt{b_{t-i}}}}{\sum_{t=0}^{T-1} \eta_t}
= 
\frac{1}{\sqrt{b}} 
\frac{
\sum_{t=0}^{T-1} \eta_t \sum_{i=0}^{t} \beta^i}
{\sum_{t=0}^{T-1} \eta_t}
\leq 
\frac{1}{(1 - \beta)\sqrt{b}}
\text{ and } 
\frac{\sum_{t=0}^{T-1} \frac{\eta_t}{\sqrt{b_t}}}{\sum_{t=0}^{T-1} \eta_t}
= \frac{1}{\sqrt{b}}.
\end{align*}
Hence, Theorem \ref{thm:1} leads to Corollary \ref{cor:1}(vii).

(viii) From $b_t = b \delta^t$, we have 
\begin{align*}
\sum_{t=0}^{T-1}\eta_t \sum_{i=0}^{t}\frac{\beta^i}{\sqrt{b_{t-i}}}
\le
\frac{1}{1-\beta}\sum_{k=0}^{T-1}\frac{\eta_k}{\sqrt{b_k}}
\le
\frac{\eta \sqrt{\delta}}{(1-\beta)(\sqrt{\delta} - 1)\sqrt{b}} 
\text{ and }
\sum_{t=0}^{T-1} \frac{\eta_t}{\sqrt{b_t}}
\leq 
\frac{\eta \sqrt{\delta}}{(\sqrt{\delta} - 1)\sqrt{b}}.
\end{align*}
Hence,
\begin{align*}
\frac{\sum_{t=0}^{T-1} \eta_t \sum_{i=0}^{t} \frac{\beta^i}{\sqrt{b_{t-i}}}}{\sum_{t=0}^{T-1} \eta_t}
\leq 
\frac{\sqrt{\delta}}{(1-\beta)(\sqrt{\delta} - 1)\sqrt{b}}
\times 
\begin{dcases}
    \frac{1 - a}{T^{1-a}-1} \text{ } &(a \in (0,1)),\\
    \frac{1}{\log T} \text{ } &(a =1)
\end{dcases}
\end{align*}
and 
\begin{align*}
\frac{\sum_{t=0}^{T-1} \frac{\eta_t}{\sqrt{b_t}}}{\sum_{t=0}^{T-1} \eta_t}
\leq
\frac{\sqrt{\delta}}{(\sqrt{\delta} - 1)\sqrt{b}}
\times 
\begin{dcases}
    \frac{1 - a}{T^{1-a}-1} \text{ } &(a \in (0,1)),\\
    \frac{1}{\log T} \text{ } &(a =1).
\end{dcases}
\end{align*}
Hence, Theorem \ref{thm:1} leads to Corollary \ref{cor:1}(viii).
\end{proof}

\end{document}